\documentclass[12pt]{article}

\usepackage[utf8]{inputenc} 
\usepackage[T1]{fontenc}    
\usepackage{hyperref}       
\usepackage{url}            
\usepackage{booktabs}       
\usepackage{amssymb}		
\usepackage{nicefrac}       
\usepackage{microtype}      
\usepackage{xcolor}         

\usepackage{amsmath}
\usepackage{graphicx}
\usepackage{amsthm}
\usepackage{euscript}
\usepackage{stmaryrd}
\usepackage[colors,nocontrol]{optsys}
\usepackage{array}

\usepackage{mathtools}
\usepackage{wasysym}
\usepackage{cases}
\usepackage{lmodern}

\newcommand{\supportFunction}{{\sigma}}
\newcommand{\SupportFunction}{\sigma}
\newcommand{\SingularValues}{s}

\newcommand{\couplingCapraNormMatrix}{\couplingCAPRA}

\newcommand{\CoordinateNormMatrix}[2]{{#1}_{(#2)}^{\mathrm{rk}}}
\newcommand{\CoordinateNormDualMatrix}[2]{{#1}_{(#2),\star}^{\mathrm{rk}}}
\GTripleNorm{NormMatrix}

\newcommand{\NormVector}{\Norm}
\newcommand{\TopNormMatrix}{\TopNorm}
\newcommand{\LEQBALLrangK}[1]{\mathbb{B}^{\leq #1}_{\NormMatrix{\cdot}}}
\newcommand{\BALLrangK}[1]{\mathbb{B}^{= #1}_{\NormMatrix{\cdot}}}
\newcommand{\LEQSPHERErangK}[1]{\SPHERE^{\leq #1}_{\NormMatrix{\cdot}}}
\newcommand{\SPHERErangK}[1]{\SPHERE^{= #1}_{\NormMatrix{\cdot}}}
\newcommand{\MATRIX}{{\cal M}}
\newcommand{\MATRIXdim}[2]{\MATRIX_{#1,#2}}
\newcommand{\MATRIXrang}[1]{{\cal M}^{=#1}}
\newcommand{\MATRIXrangLEQ}[1]{{\cal M}^{\leq #1}}
\newcommand{\transpose}[1]{#1^{\mathrm{T}}}
\newcommand{\primalMatrix}{M}
\newcommand{\fonctionprimalMatrix}{F} 
\newcommand{\dualMatrix}{N}

\newcommand{\orthogroup}[1]{{\cal O}_{#1}}
\newcommand{\Trace}{\mathrm{Tr}}
\newcommand{\rows}{m}
\newcommand{\columns}{n}
\newcommand{\minrowscolumns}{d}

\newcommand{\LocalDiagonal}{i}
\newcommand{\LocalRank}{r}
\renewcommand{\LocalIndex}{k}
\renewcommand{\LocalIndexbis}{l}

\newcommand{\integer}{t} 

\DeclareMathOperator{\rank}{rk}
\DeclareMathOperator{\diag}{diag}

\renewcommand{\bleue}[1]{#1}

\usepackage{fullpage}

\title{Rank-Based Norms, \Capra-Conjugacies\\ and the Rank Function}

\author{Paul Barbier, Jean-Philippe Chancelier,
  Michel De Lara, Valentin Paravy, \\ CERMICS, \'Ecole des Ponts, Marne-la-Vall\'ee, France}

\begin{document}
\maketitle

\begin{abstract}
  We consider the space of matrices, with given number of rows and of columns,
  equipped with the classic trace scalar product.
  With any matrix (source) norm, we associate a coupling, called Capra, between the space of matrices and
  itself. Then, we compute the Capra conjugate and biconjugate of the rank function.
  They are expressed in function of a sequence of rank-based norms, more precisely 
  generalized r-rank and dual r-rank matrix norms associated with the matrix source norm.
  We deduce a lower bound of the rank function given by a variational formula which involves
  the generalized r-rank norms. In the case of the Frobenius norm, we show that
  the rank function is equal to the variational formula. 
\end{abstract}


\section{Introduction}
\label{sec1}

%
The rank function is a well-known example of nonconvex and nonsmooth function
over matrices
(as it is not possible to cover all the references on such a large subject,
we refer the reader to a small subset
\cite{Fazel-Hindi-Boyd:2001,Fazel-Hindi-Boyd:2004}
of the literature
and to \cite{Hiriart-Urruty-Le:2013} which offers a kind of survey of the rank
function). 
In this paper, we display a variational lower bound of the rank function that
involves a sequence of suitable norms.

For this purpose, we introduce sequences of rank-based norms ---
more precisely, generalized $\LocalRank$-rank and dual $\LocalRank$-rank matrix
norms --- generated from any (source) norm.
\bleue{%
  This construction, for matrices, can also be found in
  \cite{Grussler-Giselsson:2018}, but in the case of unitarily invariant source
  norms.}
With a general source norm, we also define a coupling between the space of matrices and
itself, and we compute the biconjugate of the rank function 
under the associated conjugacy.
We deduce a lower bound variational formula for the rank function which involves
generalized $\LocalRank$-rank norms. 
Moreover, when the source norm is the Frobenius norm, we
prove that the inequality is an equality.
\bleue{%
  The conjugacy we use is not the Fenchel conjugacy.
  This latter has been used for instance in~\cite{Larsson-Olsson:2016} to obtain
  convex lower envelopes of matrix functions of the form
  \( \primalMatrix \mapsto \varphi\bp{\rank\np{\primalMatrix}} +
  \NormMatrix{\primalMatrix-\primalMatrix_0}^2 \), or in~\cite{Grussler-Giselsson:2018} to obtain
  convex lower envelopes of matrix functions of the form
  $\primalMatrix \mapsto \varphi\np{\norm{\primalMatrix}} + \delta_{\rank\np{\primalMatrix}\le r}$,
  and then get convex low rank approximation of
  optimization problems with low rank solutions. This is not the approach we
  follow in this paper: we use a new \Capra-conjugacy to analyze, and provide
  variational formulas for, the rank function, but we are not motivated (at this
  stage) by possible use in optimization under rank constraint.}

The paper is organized as follows. 
In Sect.~\ref{Rank-based norms}, we define and study rank-based norms.
In Sect.~\ref{CAPRA-conjugacies_and_the_rank_function},
we introduce \Capra-conjugacies and their relations with the rank function.

\section{Rank-based norms}
\label{Rank-based norms}

In \S\ref{Notation}, we fix notation.
In \S\ref{Definition_of_generalized_rank_and_dual_rank_matrix_norms}, we define
rank-based norms as, more precisely, 
generalized $\LocalRank$-rank and dual $\LocalRank$-rank matrix norms.
In \S\ref{The_case_of_unitarily_invariant_source_matrix_norms}, we detail 
the case of unitarily invariant source matrix norms. 

\subsection{Notation}
\label{Notation}

In all the paper, we consider two fixed positive integers
$\rows$ (number of rows) and $\columns$ (number of columns),
and we denote \( \minrowscolumns=\min\np{\rows,\columns} \).
We use the notation \(
\ic{\LocalIndex,\LocalIndexbis}=\na{\LocalIndex,\LocalIndex+1,\ldots,\LocalIndexbis-1,\LocalIndexbis}
\) for any pair of integers such that \( \LocalIndex \leq \LocalIndexbis \).
We denote by~$\MATRIXdim{\rows}{\columns}$ 
the space of real matrices with $\rows$~rows and $\columns$~columns,
by~$\rank : \MATRIXdim{\rows}{\columns} \to \NN$ the rank function (where $\NN$ is the set of non-negative integers)
and by~$\MATRIXrangLEQ{\LocalRank}$ (resp. $\MATRIXrang{\LocalRank}$) the subset of matrices of rank less
than or equal to~$\LocalRank$ (resp. equal to~$\LocalRank$).
We recall that the singular values of a matrix~$\primalMatrix \in
\MATRIXdim{\rows}{\columns}$ are the square root of the (nonnegative)
eigenvalues of the square matrix~$\transpose{\primalMatrix}\primalMatrix$, and we denote
by $ \SingularValues\np{\primalMatrix}=\sequence{\SingularValues_{\LocalDiagonal}\np{\primalMatrix}}{\LocalDiagonal\in\ic{1,\minrowscolumns}}
\in \RR^{\minrowscolumns}$ the vector
composed of the singular values of~$\primalMatrix$ arranged in nonincreasing
order, that is, 
\begin{equation}
  \SingularValues\np{\primalMatrix} \in \Cone =
  \bset{\primal\in\RR^{\minrowscolumns}}
  {\primal_{1} \geq \cdots \geq \primal_{\minrowscolumns} \geq 0 }
  = \SingularValues\np{\MATRIXdim{\rows}{\columns} }
  \eqsepv \forall \primalMatrix\in \MATRIXdim{\rows}{\columns} 
  \eqfinp 
  \label{eq:CONE}
\end{equation}
For any \( \integer\in\NN^* = \NN\setminus\na{0} \), we denote by~\( \orthogroup{\integer} \) the group of orthogonal square
$\integer\times\integer$ matrices.
It is established that, for any matrix~$\primalMatrix \in
\MATRIXdim{\rows}{\columns}$, there exists a \emph{singular value decomposition} \cite[p. 6]{Bhatia:1997}
\( \primalMatrix = U\diag(\SingularValues\np{\primalMatrix})\transpose{V} \) of the
matrix~$\primalMatrix$, where $U \in \orthogroup{{\rows}}$ and $V \in
\orthogroup{{\columns}}$.
It is also readily proven that, for any matrix~$\primalMatrix \in
\MATRIXdim{\rows}{\columns}$, for any $U \in \orthogroup{{\rows}}$ and $V \in
\orthogroup{{\columns}}$, we have that
\( \SingularValues\np{\primalMatrix}=\SingularValues\np{U\primalMatrix V} \). 

When equipped with the scalar product
$ \MATRIXdim{\rows}{\columns}^2 \ni \primalMatrix, \dualMatrix \mapsto
\Trace\np{\primalMatrix\transpose{\dualMatrix}}$, $\MATRIXdim{\rows}{\columns}$ is an
Euclidean space which is in duality with itself.
As we manipulate functions with values 
in~$\barRR = [-\infty,+\infty] $,
we adopt the Moreau \emph{lower and upper additions} \cite{Moreau:1970} 
that extend the usual addition with 
$\np{+\infty} \LowPlus \np{-\infty} = \np{-\infty} \LowPlus \np{+\infty} =
-\infty $ or with $ \np{+\infty} \UppPlus \np{-\infty} = \np{-\infty} \UppPlus \np{+\infty} =
+\infty$.
For any subset \( \Dual\subset \MATRIXdim{\rows}{\columns} \),
\( \SupportFunction_{\Dual} : \MATRIXdim{\rows}{\columns} \to \barRR\) denotes the 
\emph{support function of the subset~$\Dual$}:
\begin{equation}
  \SupportFunction_{\Dual}\np{\primalMatrix} = 
  \sup_{\dualMatrix\in\Dual} \Trace\np{\primalMatrix\transpose{\dualMatrix}}
  \eqsepv \forall \primalMatrix\in \MATRIXdim{\rows}{\columns} 
  \eqfinp
  \label{eq:support_function}
\end{equation}
A generic norm on the space~$\MATRIXdim{\rows}{\columns}$ of matrices
will be denoted by~$\NormMatrix{\cdot}$ and will be called \emph{matrix
  norm}\footnote{%
  In some books, the terminology \emph{matrix norm} is reserved for submultiplicative norms
  over matrices, which is not the case here.}.
By contrast, a generic norm on the space~$\RR^{\minrowscolumns}$  of vectors
will be denoted by~$\Norm{\cdot}$ and will be called \emph{vector  norm}.
For any matrix norm~$\NormMatrix{\cdot}$ on the
space~$\MATRIXdim{\rows}{\columns}$,
we denote by~$\BALL_{\NormMatrix{\cdot}} \subset
\MATRIXdim{\rows}{\columns}$ and~$\SPHERE_{\NormMatrix{\cdot}}\subset
\BALL_{\NormMatrix{\cdot}} \subset \MATRIXdim{\rows}{\columns}$ the associated 
unit ball and unit sphere. 
The dual norm~$\NormMatrix{\cdot}_{\star}$ of the matrix
norm~$\NormMatrix{\cdot}$ is a matrix norm on the space~$\MATRIXdim{\rows}{\columns}$,
defined by
\( \NormMatrix{\cdot}_{\star}=\SupportFunction_{\BALL_{\NormMatrix{\cdot}}} \).

\subsection{Definition of generalized $\LocalRank$-rank and dual
  $\LocalRank$-rank matrix norms}
\label{Definition_of_generalized_rank_and_dual_rank_matrix_norms}

To define rank-based norms, one could take inspiration from the following construction of
vector norms as
in~\cite[Definition~3.2]{Chancelier-DeLara:2022_CAPRA_OPTIMIZATION}
and in~\cite[Definition~3]{Chancelier-DeLara:2022_SVVA}.
Given a vector norm $\NormVector{\cdot}$ on~$\RR^d$, one can define other norms as
follows: for any subset~$K \subset \ic{1,d}$ of indices and for any vector
$\primal$,
one denotes by~$\primal_K$ the vector that coincides with~$\primal$ for the
indices in~$K$ and with null entries outside~$K$;
by taking the supremum of the norm of all these vectors~$\primal_K$,
for a cardinality~$\cardinality{K}$ of~$K$ smaller or equal to a fixed integer~$r$, one obtains
$\max\limits_{\cardinality{K} \leq r} \NormVector{\primal_{K}}$.
Such a construction indeed defines a norm, which has been studied in
\cite{Chancelier-DeLara:2022_CAPRA_OPTIMIZATION,Chancelier-DeLara:2022_SVVA}.
Unfortunately, this procedure does not work with the rank, as we illustrate
below. 

Let $\NormMatrix{\cdot}$ be the  \( \ell_1 \)-norm on the space~$\MATRIXdim{\rows}{\columns}$ of
matrices, that is, \( \NormMatrix{\primalMatrix} \) is the sum of the modules of
all the components of the matrix \( \primalMatrix\in\MATRIXdim{\rows}{\columns}
\).
Define, for any matrix $\primalMatrix\in\MATRIXdim{\rows}{\columns}$ and
$\LocalRank \in\ic{1,\minrowscolumns}$, 
$ \TopNormMatrix{\NormMatrix{\primalMatrix}}{\LocalRank}
=\sup_{X \subset \primalMatrix, \rank\np{X} \leq \LocalRank} \NormMatrix{X}$, where
$X \subset \primalMatrix$ is a shorthand for matrices~$X$ of
$\MATRIXdim{\rows}{\columns}$ for which there exists $K\subset\ic{1,\rows}$ and 
$L\subset\ic{1,\columns}$ such that $X$ coincides with $\primalMatrix$, except for
entries $X_{k,l}=0$ for $\np{k,l}\notin\produit{K}{L}$. 
We now show that the function~$\TopNormMatrix{\NormMatrix{\cdot}}{\LocalRank}$
is not a norm by contradicting the triangular inequality.
Indeed, consider $\rows= \columns= 2$ and the matrix
\(\primalMatrix = \begin{pmatrix} 1 & 1  \\ 1 & 1 \end{pmatrix}\).
As $\rank\np{\primalMatrix} = 1$, we get that 
$\TopNormMatrix{\NormMatrix{\primalMatrix}}{1} = 4$.
We can write
$\primalMatrix = 
\begin{pmatrix}
  1 & 0  \\
  0 & 1
\end{pmatrix} + 
\begin{pmatrix}
  0 & 1  \\
  1 & 0
\end{pmatrix}$ and we easily get that
$\TopNormMatrix{\BNormMatrix{\begin{pmatrix} 1 & 0  \\ 0 & 1 \end{pmatrix}}}{1} = 
\TopNormMatrix{\BNormMatrix{\begin{pmatrix}  0 & 1  \\ 1 & 0 \end{pmatrix}}}{1} = 1$.
However, the triangular inequality does not hold true as we have that
$4 = \TopNormMatrix{\NormMatrix{\primalMatrix}}{1} >
\TopNormMatrix{\BNormMatrix{\begin{pmatrix} 1 & 0  \\ 0 & 1 \end{pmatrix}}}{1} +
\TopNormMatrix{\BNormMatrix{\begin{pmatrix} 0 & 1  \\ 1 & 0 \end{pmatrix}}}{1} = 1 + 1 = 2$. 
                                                                                                     
This is why we turn to the following definition, 
inspired by the properties of dual coordinate-$\LocalRank$ vector norms
in \cite[Equation~(16), Proposition~3.3]{Chancelier-DeLara:2022_CAPRA_OPTIMIZATION}.
This construction, for matrices, can also be found in
\cite{Grussler-Giselsson:2018} in the unitarily invariant norm case
(see the discussion at the beginning of
\S\ref{The_case_of_unitarily_invariant_source_matrix_norms}).

\begin{proposition}
  \label{pr:CoordinateNormMatrix}
  Let $\NormMatrix{\cdot}$ be a norm on the space~$\MATRIXdim{\rows}{\columns}$ of
  matrices.
  We denote by~$\BALL_{\NormMatrix{\cdot}}$ and~$\SPHERE_{\NormMatrix{\cdot}}$ the associated
  unit ball and unit sphere, as well as, for any \( \LocalRank \in \ic{0,\minrowscolumns}\),
  \begin{equation}
    \LEQBALLrangK{\LocalRank}=
    \BALL_{\NormMatrix{\cdot}} \cap \MATRIXrangLEQ{\LocalRank}
    \eqsepv
    \BALLrangK{\LocalRank}=\BALL_{\NormMatrix{\cdot}} \cap \MATRIXrang{\LocalRank} \eqsepv
    \LEQSPHERErangK{\LocalRank}=\SPHERE_{\NormMatrix{\cdot}} \cap \MATRIXrangLEQ{\LocalRank} \eqsepv 
    \SPHERErangK{\LocalRank}=\SPHERE_{\NormMatrix{\cdot}} \cap \MATRIXrang{\LocalRank}
    \eqfinp  
    \label{eq:rank_balls_sphere}
  \end{equation}
  The following expressions~$\CoordinateNormDualMatrix{\NormMatrix{\cdot}}{\LocalRank}$ define
  a nondecreasing sequence \(
  \sequence{\CoordinateNormDualMatrix{\NormMatrix{\cdot}}{\LocalRank}}{\LocalRank\in\ic{1,\minrowscolumns}} \)
  of norms on~$\MATRIXdim{\rows}{\columns}$
  \begin{equation}
    \CoordinateNormDualMatrix{\NormMatrix{\dualMatrix}}{\LocalRank}
    =\SupportFunction_{\LEQBALLrangK{\LocalRank}}\np{\dualMatrix}
    = \sup_{\primalMatrix\in \LEQBALLrangK{\LocalRank}}\Trace\np{\primalMatrix\transpose{\dualMatrix}}
    \eqsepv \forall \dualMatrix\in \MATRIXdim{\rows}{\columns}
    \eqsepv \forall \LocalRank \in \ic{1,\minrowscolumns}
    \eqfinv 
    \label{eq:CoordinateNormMatrix}
  \end{equation}
  which satisfy
  \begin{equation}
    \CoordinateNormDualMatrix{\NormMatrix{\cdot}}{\LocalRank}
    = \SupportFunction_{\BALLrangK{\LocalRank}}
    =  \SupportFunction_{\LEQSPHERErangK{\LocalRank}}
    =  \SupportFunction_{\SPHERErangK{\LocalRank}}
    \eqsepv \forall \LocalRank \in \ic{1,\minrowscolumns}
    \eqfinp
    \label{eq:CoordinateNormMatrix=sup_RankTopNorm}      
  \end{equation}
\end{proposition}        
\begin{proof}
  The  sequence 
  \(
  \bseqa{\CoordinateNormDualMatrix{\NormMatrix{\cdot}}{\LocalRank}}{\LocalRank\in\ic{1,\minrowscolumns}}
  \)
  in~\eqref{eq:CoordinateNormMatrix} is nondecreasing
  since the sequence \(
  \bseqa{\LEQBALLrangK{\LocalRank}}{\LocalRank \in \ic{1,\minrowscolumns}} \) of unit balls
  in~\eqref{eq:rank_balls_sphere}  is nondecreasing,
  as so is the sequence $
  \nseqa{\MATRIXrangLEQ{\LocalRank} }{\LocalRank \in\ic{1,\minrowscolumns}}$. 

  First, we prove that
  $\SupportFunction_{\SPHERErangK{\LocalRank}}=\SupportFunction_{\LEQSPHERErangK{\LocalRank}}$,
  for $ \LocalRank \in \ic{1,\minrowscolumns}$. 
  For this purpose, we show that $
  \overline{\SPHERErangK{\LocalRank}}=\LEQSPHERErangK{\LocalRank}$,
  where $ \overline{\,\cdot\,} $ denotes the topological closure.
  The inclusion $\overline{\SPHERErangK{\LocalRank}}
  \subset\LEQSPHERErangK{\LocalRank}  $ is straightforward because it is well
  known \cite[Theorem~2]{Hiriart-Urruty-Le:2013} that $\overline{\MATRIXrang{\LocalRank}}=\MATRIXrangLEQ{\LocalRank}$, from
  which we deduce that 
  $\overline{\SPHERErangK{\LocalRank}}=
  \overline{\SPHERE_{\NormMatrix{\cdot}}\cap\MATRIXrang{\LocalRank}}
  \subset \overline{\SPHERE_{\NormMatrix{\cdot}}}\cap \overline{\MATRIXrang{\LocalRank}}
  =\SPHERE_{\NormMatrix{\cdot}}\cap\MATRIXrangLEQ{\LocalRank}$, 
  the inclusion being a property of the topological closure.

  To prove the reverse inclusion
  $\SPHERE_{\NormMatrix{\cdot}}\cap\MATRIXrangLEQ{\LocalRank}
  \subset \overline{\SPHERErangK{\LocalRank}} $,
  we consider $ \primalMatrix\in
  \SPHERE_{\NormMatrix{\cdot}}\cap\MATRIXrangLEQ{\LocalRank}$. 
  As $\overline{\MATRIXrang{\LocalRank}}=\MATRIXrangLEQ{\LocalRank}$, there exists
  a sequence $\{\primalMatrix_n\}_{n\in \NN}$ in $\MATRIXrang{\LocalRank}$ such that
  $\primalMatrix_n\rightarrow \primalMatrix$ when $n \rightarrow +\infty$. Since
  $\primalMatrix\in\SPHERE_{\NormMatrix{\cdot}}$, we can always suppose that $\primalMatrix_n \ne
  0$, for all $n\in \NN$. Therefore $\frac{\primalMatrix_n}{\NormMatrix{\primalMatrix_n}}$ is well
  defined, and when $n \rightarrow +\infty $ we have that 
  $\frac{\primalMatrix_n}{\NormMatrix{\primalMatrix_n}} \rightarrow
  \frac{\primalMatrix}{\NormMatrix{\primalMatrix}}=\primalMatrix$ since
  $\primalMatrix\in\SPHERE_{\NormMatrix{\cdot}}$. Now, for all $n\in \NN$, on the one hand,
  $\frac{\primalMatrix_n}{\NormMatrix{\primalMatrix_n}}\in\MATRIXrang{\LocalRank}$ and, on the other
  hand, $\frac{\primalMatrix_n}{\NormMatrix{\primalMatrix_n}}\in\SPHERE_{\NormMatrix{\cdot}}$. As a
  consequence, we get that the sequence $ \bseqa{\frac{\primalMatrix_n}{\NormMatrix{\primalMatrix_n}}}{n\in \NN}
  \in \SPHERE_{\NormMatrix{\cdot}} \cap \MATRIXrang{\LocalRank}$, and we conclude
  that the limit of the sequence $\primalMatrix\in\overline{\SPHERErangK{\LocalRank}}$. Thus, we have proven that
  $ \overline{\SPHERErangK{\LocalRank}}=\LEQSPHERErangK{\LocalRank} $, hence
  that
  $\SupportFunction_{\LEQSPHERErangK{\LocalRank}}=
  \SupportFunction_{\overline{\SPHERErangK{\LocalRank}}}= 
  \SupportFunction_{{\SPHERErangK{\LocalRank}}}$ by
  \cite[Proposition~7.13]{Bauschke-Combettes:2017}.

  Second, we prove that 
  $ \SupportFunction_{\LEQBALLrangK{\LocalRank}}
  =  \SupportFunction_{\LEQSPHERErangK{\LocalRank}} $. 
  It is readily established that
  $ \LEQSPHERErangK{\LocalRank} \subset \LEQBALLrangK{\LocalRank}          
  \subset \convexhull\LEQSPHERErangK{\LocalRank} $
  (the convex hull of~$\LEQSPHERErangK{\LocalRank}$)
  as any point in $\LEQBALLrangK{\LocalRank}= \BALL_{\NormMatrix{\cdot}} \cap \MATRIXrangLEQ{\LocalRank}$ is the convex
  combination of a point and its opposite in $\LEQSPHERErangK{\LocalRank}
  = \SPHERE_{\NormMatrix{\cdot}} \cap \MATRIXrangLEQ{\LocalRank}$.
  Therefore, by property \cite[Proposition~7.13]{Bauschke-Combettes:2017} of the
  support function~\eqref{eq:support_function}, 
  we get that $ \SupportFunction_{\LEQBALLrangK{\LocalRank}}
  =  \SupportFunction_{\LEQSPHERErangK{\LocalRank}} $.
  By the same reasoning, we also obtain that $ \SupportFunction_{\BALLrangK{\LocalRank}}=\SupportFunction_{\SPHERErangK{\LocalRank}}
  $ which, combined with the first part, gives $ \SupportFunction_{\LEQBALLrangK{\LocalRank}}
  =  \SupportFunction_{\LEQSPHERErangK{\LocalRank}} =\SupportFunction_{{\SPHERErangK{\LocalRank}}}= 
  \SupportFunction_{\BALLrangK{\LocalRank}}
  $.
  
  Third, we prove that~\eqref{eq:CoordinateNormMatrix} defines norms.
  We consider a fixed $ \LocalRank \in \ic{1,\minrowscolumns} $.
  As the set $\LEQBALLrangK{\LocalRank}$ is easily seen to be bounded and
  symmetric, $    \CoordinateNormDualMatrix{\NormMatrix{\cdot}}{\LocalRank}
  =\SupportFunction_{\LEQBALLrangK{\LocalRank}} $ is a
  1-homogeneous subadditive function with values in~$[0,+\infty[$. 
  It remains to
  prove that, for any $\dualMatrix\in\MATRIXdim{\rows}{\columns}$,
  $ \SupportFunction_{\LEQBALLrangK{\LocalRank}}\np{\dualMatrix}=0 \Leftrightarrow \dualMatrix=0$.  For this
  purpose, we consider a matrix $\dualMatrix
  \in \MATRIXdim{\rows}{\columns}$ which satisfies
  $ \SupportFunction_{\LEQBALLrangK{\LocalRank}}\np{\dualMatrix}=0 $, 
  and we prove that $\dualMatrix = 0$.
  We consider the singular value decomposition
  $ \dualMatrix = U\diag(\SingularValues\np{\dualMatrix})\transpose{V} $ of the
  matrix~$\dualMatrix$. 
  Defining $ \primalMatrix = U\diag\bp{\SingularValues_1\np{\dualMatrix},0,\ldots,0}
  \transpose{V} $,
  the matrix~$ \primalMatrix $ has rank less than or equal to~1.
  Thus, we obtain that 
  \begin{equation*}
    \module{\SingularValues_{1}\np{\dualMatrix}}^2  =
    \module{\SingularValues_{1}\np{\primalMatrix}}^2  =
    \Trace\np{\primalMatrix\transpose{\dualMatrix}}
    \leq  \NormMatrix{\primalMatrix} 
    \sup_{\primalMatrix'\in
      \LEQSPHERErangK{\LocalRank}}\Trace(\transpose{\primalMatrix'}\dualMatrix)
    =\NormMatrix{\primalMatrix} 
    \SupportFunction_{\LEQBALLrangK{\LocalRank}}\np{\dualMatrix}=0
    \eqfinv 
  \end{equation*}
  hence that $
  \SingularValues_1\np{\dualMatrix}=0  $.
  This implies that all the singular values of~$\dualMatrix$ are null because
  $\SingularValues_1\np{\dualMatrix}$ is the largest one. Hence, we get that
  $\dualMatrix = 0$. 
  
  This ends the proof. 
\end{proof}

Now, we define rank-based norms as follows. 
\begin{definition}
  Let $\NormMatrix{\cdot}$ be a norm on the
  space~$\MATRIXdim{\rows}{\columns}$ of matrices,
  that we call \emph{source (matrix) norm}.
  The matrix norms in the nondecreasing sequence $
  \sequence{\CoordinateNormDualMatrix{\NormMatrix{\cdot}}{\LocalRank}}{\LocalRank \in\ic{1,\minrowscolumns}} $,
  given by Proposition~\ref{pr:CoordinateNormMatrix}, are called
  \emph{generalized dual $\LocalRank$-rank matrix norms}.
  By taking their dual norms
  $ \CoordinateNormMatrix{\NormMatrix{\cdot}}{\LocalRank}
  =\bp{\CoordinateNormDualMatrix{\NormMatrix{\cdot}}{\LocalRank}}_\star$, we obtain a nonincreasing
  sequence $
  \sequence{\CoordinateNormMatrix{\NormMatrix{\cdot}}{\LocalRank}}{\LocalRank \in\ic{1,\minrowscolumns}} $ of norms on~$\MATRIXdim{\rows}{\columns}$ called
  \emph{generalized $\LocalRank$-rank matrix norms}.
  \label{de:CoordinateNormMatrix}
\end{definition}
Notice that, by~\eqref{eq:CoordinateNormMatrix} for $
\LocalRank=\minrowscolumns $, and then by taking the dual norms, we get that
\begin{equation}
  \CoordinateNormDualMatrix{\NormMatrix{\cdot}}{1}
  \leq \cdots \leq 
  \CoordinateNormDualMatrix{\NormMatrix{\cdot}}{\minrowscolumns} =
  \NormMatrix{\cdot}_\star
  \mtext{ and }
  \CoordinateNormMatrix{\NormMatrix{\cdot}}{1}
  \geq \cdots \geq 
  \CoordinateNormMatrix{\NormMatrix{\cdot}}{\minrowscolumns} =
  \NormMatrix{\cdot}
  \eqfinp 
\end{equation}

When the source norm~$\NormMatrix{\cdot}$  is unitarily invariant (see \S\ref{The_case_of_unitarily_invariant_source_matrix_norms}),
the norms above have been introduced and studied in~\cite{Grussler-Giselsson:2018}.
Thus, we provide an extension (hence, the term \emph{generalized})
of the so-called \emph{rank constrained dual norm}
in~\cite[Equation~(7)]{Grussler-Giselsson:2018} to
generalized dual $\LocalRank$-rank matrix norm \( \CoordinateNormDualMatrix{\NormMatrix{\cdot}}{\LocalRank} \)
in Equation~\eqref{eq:CoordinateNormMatrix} in Proposition~\ref{pr:CoordinateNormMatrix},
and of the so-called \emph{low-rank inducing norm}
in~\cite[Equation~(8)]{Grussler-Giselsson:2018} to
generalized $\LocalRank$-rank matrix norm
\( \CoordinateNormMatrix{\NormMatrix{\cdot}}{\LocalRank} \) in 
Definition~\ref{de:CoordinateNormMatrix}.
%
This extension is justified as our main result --- namely, a lower bound for the
rank function in Theorem~\ref{th:variational_lower_bound_of_the_rank_function} --- 
holds for any source norm, unitarily invariant or not, and 
involves generalized $\LocalRank$-rank matrix norms.
Some common norms are not unitarily invariant, such as the $\ell_p$-norm for $p
\ne 2$ (including the supremum norm when $p=\infty$).
However, we have not been able to obtain explicit formulas for 
generalized $\LocalRank$-rank matrix norms in these special non unitarily
invariant cases.

\subsection{The case of unitarily invariant source matrix norms}
\label{The_case_of_unitarily_invariant_source_matrix_norms}

As just said, our main result (variational lower bound of the rank function)
does not require unitarily invariant norms.
However, we devote
this~\S\ref{The_case_of_unitarily_invariant_source_matrix_norms}
to unitarily invariant source matrix norms for two reasons:
to stress proximity and difference with \cite{Grussler-Giselsson:2018};
to provide a special case where the inequality in the forthcoming
Theorem~\ref{th:variational_lower_bound_of_the_rank_function}
is an equality. 

In~\S\ref{Background_on_unitarily_invariant_matrix_norms},
we provide background on unitarily invariant matrix norms.
In~\S\ref{Links_with_coordinate_norms_and_the_lzeropseudonorm},
we make the link between
generalized $\LocalRank$-rank and dual $\LocalRank$-rank matrix norms, on the
one hand, and generalized coordinate and dual coordinate-$\LocalRank$ norms and
the \lzeropseudonorm, on the other hand.

\subsubsection{Background on unitarily invariant matrix norms}
\label{Background_on_unitarily_invariant_matrix_norms}

We recall that a \emph{unitarily invariant norm}
on~$\MATRIXdim{\rows}{\columns}$ is a matrix norm such that
$ \NormMatrix{U \primalMatrix V}=\NormMatrix{\primalMatrix} $, for any matrix
$ \primalMatrix \in \MATRIXdim{\rows}{\columns} $ and orthogonal matrices $U \in \orthogroup{{\rows}}$,
$V \in \orthogroup{{\columns}}$.

We recall that a \emph{symmetric absolute norm} is a vector norm~$\NormVector{\cdot}$ on~$\RR^{\minrowscolumns}$ which satisfies the following properties:
$\NormVector{\cdot}$ is \emph{absolute} in the sense that 
$\NormVector{~\abs{\primal}~} =\NormVector{\primal} $, for any $ \primal\in\RR^{\minrowscolumns}
$, where $
\abs{\primal}=\np{\abs{\primal_1},\ldots,\abs{\primal_{\minrowscolumns}}} $;
$\NormVector{\cdot}$ is \emph{symmetric} (or \emph{permutation invariant}), that is, 
$\NormVector{\np{\primal_{\nu(1)},\ldots,\primal_{\nu(\minrowscolumns)}}}=
\NormVector{\np{\primal_1,\ldots,\primal_{\minrowscolumns}}} $, for any $ \primal\in\RR^{\minrowscolumns}
$ and for any permutation~$ \nu $ of the indices in~$ \ic{1,\minrowscolumns}
$.
In the literature, a symmetric absolute norm is also often called a
\emph{symmetric gauge function} (this is the vocabulary used in
\cite{Grussler-Giselsson:2018}). 
%
%
These two notions are linked by the following property (see \cite[Theorem~IV.2.1]{Bhatia:1997}).

\begin{proposition}[Von Neumann]
  \label{pr:unitarilyinvariantnormequivalence}              
  A norm $\NormMatrix{\cdot} $ on the space~$
  \MATRIXdim{\rows}{\columns} $ of matrices is
  unitarily invariant if and only if there exists
  a symmetric absolute norm~$\NormVector{\cdot}$ on~$\RR^{\minrowscolumns}$
  such that         
  \begin{equation}
    \NormMatrix{\cdot}= \NormVector{\cdot} \circ\SingularValues 
    \mtext{ that is, }
    \NormMatrix{\primalMatrix}= \NormVector{ 
      \bp{\SingularValues_1\np{\primalMatrix},\ldots,\SingularValues_{\minrowscolumns}\np{\primalMatrix}} }
    \eqfinv \forall \primalMatrix \in \MATRIXdim{\rows}{\columns}
    \eqfinp
    \label{eq:unitarilyinvariantnormequivalence}
  \end{equation}
  In that case, one has the following relation between dual norms
  \begin{equation}
    \NormMatrix{\cdot}_\star=
    \NormVector{\cdot}_\star\circ\SingularValues
    \eqfinp
    \label{eq:unitarilyinvariantnormequivalence_dual}
  \end{equation}
\end{proposition}
We call Equations~\eqref{eq:unitarilyinvariantnormequivalence}
and \eqref{eq:unitarilyinvariantnormequivalence_dual}
\emph{factorization equations} as the unitarily invariant matrix norms $\NormMatrix{\cdot} $
and $\NormMatrix{\cdot}_\star$ are factorized by means of the symmetric absolute
vector norms $\NormVector{\cdot} $ and $\NormVector{\cdot}_\star$.
%
The proof relies on the so-called \emph{Von Neumann inequality trace theorem}
\cite{KyFan-Hoffman:1955}: 
\begin{equation}
  \sup_{U \in \orthogroup{{\rows}}, V \in \orthogroup{{\columns}}}
  \Trace\np{U \primalMatrix V\transpose{\dualMatrix}}
  = \proscal{\SingularValues\np{\primalMatrix}}{\SingularValues\np{\dualMatrix}}
  \eqfinv \forall \primalMatrix, \dualMatrix \in \MATRIXdim{\rows}{\columns}
  \eqfinv
  \label{eq:VonNeumannTheorem}
\end{equation}
where $ \proscal{\cdot}{\cdot} $ is the scalar product
on~$\RR^{\minrowscolumns}$.

\subsubsection{Links with generalized coordinate-$\LocalRank$ norms and the
  \lzeropseudonorm}
\label{Links_with_coordinate_norms_and_the_lzeropseudonorm}

In~\cite[Lemma~3]{Grussler-Giselsson:2018}, it is shown that,
when the source norm~$\NormMatrix{\cdot}$ is unitarily invariant,
the rank constrained dual norms (corresponding to the generalized $\LocalRank$-rank matrix norms)
are unitarily invariant,
and a factorization equation like~\eqref{eq:unitarilyinvariantnormequivalence_dual}
is given in \cite[Equation~(9)]{Grussler-Giselsson:2018}. As a consequence, the low-rank inducing norms
(corresponding to the generalized dual $\LocalRank$-rank matrix norms) are
also unitarily invariant, and a factorization equation like~\eqref{eq:unitarilyinvariantnormequivalence}
is given in \cite[Equation~(10)]{Grussler-Giselsson:2018}.

%
In Proposition~\ref{pr:Links_with_coordinate_norms_and_the_lzeropseudonorm},
we will complete this result by providing additional characterizations of the
factorization equations~\eqref{eq:unitarilyinvariantnormequivalence}
and \eqref{eq:unitarilyinvariantnormequivalence_dual} for
the generalized $\LocalRank$-rank matrix norms
and the generalized dual $\LocalRank$-rank matrix norms. 

For this purpose, we recall that the so-called \emph{\lzeropseudonorm} on~$ \RR^\minrowscolumns $ is the function
$ \lzero : \RR^\minrowscolumns \to \ic{0,\minrowscolumns} $
defined by 
\begin{equation}
  \lzero\np{\primal} 
  = \textrm{number of nonzero components of } \primal
  \eqsepv \forall \primal \in \RR^\minrowscolumns
  \eqfinp 
  \label{eq:pseudo_norm_l0}  
\end{equation}
It is clear that rank and \lzeropseudonorm\ are related through the relation
\begin{equation}
  \rank\np{\primalMatrix}
  =
  \lzero(\SingularValues\np{\primalMatrix})
  \eqsepv \forall \primalMatrix \in \MATRIXdim{\rows}{\columns}
  \eqfinp
  \label{eq:RankEqualL0Sigma}
\end{equation}

In \cite[Definition~2]{Chancelier-DeLara:2022_SVVA}
(see also \cite[Definition~3.2]{Chancelier-DeLara:2022_CAPRA_OPTIMIZATION})
we introduce, for any vector norm~$\NormVector{\cdot}$ on~$\RR^{\minrowscolumns}$,
the sequence $ \sequence{\CoordinateNorm{\NormVector{\cdot}}{\LocalRank}}{\LocalRank \in \ic{1,\minrowscolumns}} $ 
of \emph{generalized coordinate-$\LocalRank$ norms} on~$\RR^{\minrowscolumns}$,
and the sequence $ \sequence{\CoordinateNormDual{\NormVector{\cdot}}{\LocalRank}}{\LocalRank \in \ic{1,\minrowscolumns}} $ 
of \emph{generalized dual coordinate-$\LocalRank$ norms}, their dual norms.
We do not detail their definition as we will only need the forthcoming 
characterization~\eqref{eq:dual_coordinate_norm}:
the norms $ \CoordinateNormDual{\NormVector{\cdot}}{\LocalRank} $,
for any $ \LocalRank \in \ic{1,\minrowscolumns} $,
are related to the \lzeropseudonorm\ by means of 
\begin{subequations}
  \begin{align}
    \text{the level sets}\qquad 
    \LevelSet{\lzero}{\LocalRank} 
    &= 
      \defset{ \primal \in \RR^\minrowscolumns }{ \lzero\np{\primal} \leq \LocalRank }
      \eqsepv \forall \LocalRank \in \ic{0,\minrowscolumns}
      \eqfinv
      \label{eq:pseudonormlzero_level_set}
    \\
    \text{and the level curves}\qquad 
    \LevelCurve{\lzero}{\LocalRank} 
    &= 
      \defset{ \primal \in \RR^\minrowscolumns }{ \lzero\np{\primal} = \LocalRank }
      \eqsepv \forall \LocalRank \in \ic{0,\minrowscolumns} 
      \eqfinv
      \label{eq:pseudonormlzero_level_curve}
  \end{align}
\end{subequations}
as it is proven in
\cite[Equation~(16)]{Chancelier-DeLara:2022_CAPRA_OPTIMIZATION} that, 
for any $ \LocalRank \in \ic{1,\minrowscolumns} $, the generalized dual coordinate-$\LocalRank$ norm satisfies
\begin{equation}
  \CoordinateNormDual{\NormVector{\cdot}}{\LocalRank} 
  =
  \sigma_{ \LevelSet{\lzero}{\LocalRank} \cap \TripleNormSphere_{\NormVector{\cdot}}} 
  =
  \sigma_{ \LevelCurve{\lzero}{\LocalRank} \cap \TripleNormSphere_{\NormVector{\cdot}} }
  \eqfinv
  \label{eq:dual_coordinate_norm}
\end{equation}
where $ \TripleNormSphere_{\NormVector{\cdot}} \subset \RR^{\minrowscolumns} $ denotes the unit
sphere of the norm~$\NormVector{\cdot}$.
The expression~\eqref{eq:dual_coordinate_norm} is reminiscent,
using~\eqref{eq:rank_balls_sphere}, 
of the property~\eqref{eq:CoordinateNormMatrix=sup_RankTopNorm} of the
generalized rank-based norms.

In \cite[Definition~3]{Chancelier-DeLara:2022_SVVA}
we introduce, for any vector norm~$\NormVector{\cdot}$ on~$\RR^{\minrowscolumns}$,
the sequence $ \sequence{\SupportDualNorm{\NormVector{\cdot}}{\LocalRank}}{\LocalRank \in \ic{1,\minrowscolumns}} $ 
of \emph{generalized $\LocalRank$-support dual~norms}
and 
the sequence $ \sequence{\TopDualNorm{\NormVector{\cdot}}{\LocalRank}}{\LocalRank \in \ic{1,\minrowscolumns}} $ 
of \emph{generalized top-$\LocalRank$ dual~norms}.
As with generalized coordinate-$\LocalRank$ norms, we do not detail their
definition.
However, we recall their expression when $
\NormVector{\cdot}=\NormVector{\cdot}_{\ell_p} $ is the $\ell_p$ norm. 
We establish in \cite[Table~1]{Chancelier-DeLara:2022_SVVA} that 
the associated generalized coordinate-$\LocalRank$ norm 
$ \CoordinateNorm{\NormVector{\cdot}}{\LocalRank} $ is
the \emph{\lpsupportnorm{p}{\LocalRank}}~$ \LpSupportNorm{\primal}{p}{\LocalRank} $,
and the generalized dual coordinate-$\LocalRank$ norm 
$ \CoordinateNormDual{\NormVector{\cdot}}{\LocalRank} $
is the \emph{\lptopnorm{q}{\LocalRank}}~$ \LpTopNorm{\cdot}{q}{\LocalRank} $,
where $ 1/p + 1/q = 1 $. 
For $ \dual \in \RR^d $, letting $\nu$ denote a permutation of $ \{1,\ldots,\minrowscolumns\} $ such that
$ \module{ \dual_{\nu(1)} } \geq \module{ \dual_{\nu(2)} } 
\geq \cdots \geq \module{ \dual_{\nu(\minrowscolumns)} } $,
we have that 
$ \LpTopNorm{\dual}{q}{\LocalRank}=\bp{ \sum_{l=1}^{\LocalRank} \module{ \dual_{\nu(l)} }^q}^{\frac{1}{q}} $.

%

%
In the next
Proposition~\ref{pr:Links_with_coordinate_norms_and_the_lzeropseudonorm},
we show relationships between, on the one hand, the four above sequences
--- generalized coordinate-$\LocalRank$ norms,
generalized dual coordinate-$\LocalRank$ norms,
generalized $\LocalRank$-support dual~norms,
generalized $\LocalRank$-support dual~norms ---
of vector norms (related to the \lzeropseudonorm\
\cite{Chancelier-DeLara:2022_SVVA}) 
and, on the other hand, generalized $\LocalRank$-rank matrix norms
and the generalized dual $\LocalRank$-rank matrix norms
(related to the rank function \cite{Grussler-Giselsson:2018}), 
through factorization equations like ~\eqref{eq:unitarilyinvariantnormequivalence}
  and \eqref{eq:unitarilyinvariantnormequivalence_dual}. 
As discussed at the beginning
of~\S\ref{Links_with_coordinate_norms_and_the_lzeropseudonorm},
these relationships are new (in comparison with
\cite[Lemma~3]{Grussler-Giselsson:2018}). 
  
\begin{proposition}
  \label{pr:Links_with_coordinate_norms_and_the_lzeropseudonorm}
  When the source norm~$\NormMatrix{\cdot}$ on~$ \MATRIXdim{\rows}{\columns} $ is unitarily invariant,
  with associated symmetric absolute norm~$\NormVector{\cdot}$ on $\RR^{\minrowscolumns}$
  as in Proposition~\ref{pr:unitarilyinvariantnormequivalence},
  then
  the generalized $\LocalRank$-rank matrix norms
  $ \sequence{\CoordinateNormMatrix{\NormMatrix{\cdot}}{\LocalRank}}{\LocalRank
    \in\ic{1,\minrowscolumns}} $
  and 
  the generalized dual $\LocalRank$-rank matrix norms
$\sequence{\CoordinateNormDualMatrix{\NormMatrix{\cdot}}{\LocalRank}}{\LocalRank
  \in\ic{1,\minrowscolumns}} $
(see Definition~\ref{de:CoordinateNormMatrix})
are unitarily invariant and
the factorization equations~\eqref{eq:unitarilyinvariantnormequivalence}
  and \eqref{eq:unitarilyinvariantnormequivalence_dual} are given by
  \begin{subequations}
    \begin{align}
      \CoordinateNormMatrix{\NormMatrix{\cdot}}{\LocalRank}
      &=
        \CoordinateNorm{\NormVector{\cdot}}{\LocalRank} \circ\SingularValues
        = \SupportDualNorm{\NormVector{\cdot}}{\LocalRank} \circ\SingularValues
        \eqsepv
        \forall \LocalRank \in \ic{1,\minrowscolumns}
        \eqfinv
        \label{eq:rankmatrixnormstopnorm_a}
      \\
      \CoordinateNormDualMatrix{\NormMatrix{\cdot}}{\LocalRank}
      &=
        \CoordinateNormDual{\NormVector{\cdot}}{\LocalRank} \circ\SingularValues
        =\TopDualNorm{\NormVector{\cdot}}{\LocalRank} \circ\SingularValues
        \eqsepv
        \forall \LocalRank \in \ic{1,\minrowscolumns}
        \eqfinp
        \label{eq:rankmatrixnormstopnorm_b}
    \end{align}
    \label{eq:rankmatrixnormstopnorm}
  \end{subequations}
\end{proposition}

\begin{proof}
  The first part of the Proposition --- the generalized $\LocalRank$-rank matrix norms
  $ \sequence{\CoordinateNormMatrix{\NormMatrix{\cdot}}{\LocalRank}}{\LocalRank
    \in\ic{1,\minrowscolumns}} $
  and 
  the generalized dual $\LocalRank$-rank matrix norms
$\sequence{\CoordinateNormDualMatrix{\NormMatrix{\cdot}}{\LocalRank}}{\LocalRank
  \in\ic{1,\minrowscolumns}} $
are unitarily invariant --- can be found
  in~\cite[Lemma~3]{Grussler-Giselsson:2018}.
This is why,   we now turn to our contribution, namely
Equation~\eqref{eq:rankmatrixnormstopnorm}.

  We suppose that the norm~$\NormMatrix{\cdot}$ is unitarily invariant on
  $\MATRIXdim{\rows}{\columns}$ and that 
  $ \NormVector{\cdot} $ is the associated symmetric absolute norm.
  %
  For any $ \dualMatrix \in \MATRIXdim{\rows}{\columns} $,
  we have that\footnote{%
    The proof starts like in \cite[Proof to Lemma~3, A.1]{Grussler-Giselsson:2018}, but then goes on a different direction as we
    explicitely introduce the \lzeropseudonorm.}
  \begin{align*}
    \CoordinateNormDualMatrix{\NormMatrix{\dualMatrix}}{\LocalRank}
    &=
      \sup_{\NormMatrix{\primalMatrix} = 1 \eqsepv
      \rank\np{\primalMatrix} \leq \LocalRank}
      \Trace\np{\primalMatrix\transpose{\dualMatrix}}
      \tag{by definition~\eqref{eq:CoordinateNormMatrix} and property~\eqref{eq:CoordinateNormMatrix=sup_RankTopNorm}}
    \\
    &=
      \sup_{\NormMatrix{\primalMatrix} = 1, \rank\np{\primalMatrix} \leq
      \LocalRank,  U \in \orthogroup{{\rows}} ,  V \in
      \orthogroup{{\columns}}} \Trace\np{U \primalMatrix V\transpose{\dualMatrix}}
      \intertext{by change of variable $ \primalMatrix \to U\primalMatrix V $,
      and using the properties that 
      $ \NormMatrix{U \primalMatrix V} =\NormMatrix{\primalMatrix} $
      and that 
      $ \rank\np{U \primalMatrix V} =\rank\np{\primalMatrix} $}
    &=
      \sup_{\NormMatrix{\primalMatrix} = 1, \rank\np{\primalMatrix} \leq
      \LocalRank}
      \Ba{ 
      \sup_{U \in \orthogroup{{\rows}} ,  V \in \orthogroup{{\columns}}}
      \Trace\np{U \primalMatrix V\transpose{\dualMatrix}} }
    \\
    &=
      \sup_{\NormMatrix{\primalMatrix} = 1, \rank\np{\primalMatrix} \leq \LocalRank}
      \proscal{\SingularValues\np{\dualMatrix}}{\SingularValues\np{\primalMatrix}}
      \tag{using Von Neumann inequality trace theorem~\eqref{eq:VonNeumannTheorem}} 
    \\
    &=
      \sup_{\NormVector{\SingularValues\np{\primalMatrix}} = 1, \lzero\np{\SingularValues\np{\primalMatrix}}\leq\LocalRank}
      \proscal{\SingularValues\np{\dualMatrix}}{\SingularValues\np{\primalMatrix}}
      \tag{by~\eqref{eq:unitarilyinvariantnormequivalence} and~\eqref{eq:RankEqualL0Sigma}}
    \\
    &=
      \sup_{\NormVector{\primal} = 1, \lzero\np{\primal}\leq\LocalRank, \primal\in\Cone}
      \proscal{\SingularValues\np{\dualMatrix}}{\primal}
      \intertext{as easily seen from the definition~\eqref{eq:CONE}
      of the cone~$\Cone= \SingularValues\np{\MATRIXdim{\rows}{\columns} }$,
      which is in one-to-one correspondence with the image of the singular values mapping~$\SingularValues$ } 
    &=
      \sup_{\NormVector{\primal} = 1, \lzero\np{\primal}\leq\LocalRank}
      \proscal{\SingularValues\np{\dualMatrix}}{\primal}
      \intertext{because $ \SingularValues\np{\dualMatrix} \in  \Cone $, hence the supremum is
      achieved on the cone~$\Cone$ by the well-known Hardy-Littlewood-P\'olya
      rearrangement inequality}
    &=
      \supportFunction_{\lzero^{\leq\LocalRank}\cap\SPHERE_{\NormVector{\cdot}}}
      \np{ \SingularValues\np{\dualMatrix} }
      \intertext{by definition~\eqref{eq:pseudonormlzero_level_set} of 
      the level sets~$ \LevelSet{\lzero}{\LocalRank} $,
      and as $ \TripleNormSphere_{\NormVector{\cdot}} \subset \RR^{\minrowscolumns} $ is the unit
      sphere of the norm~$\NormVector{\cdot}$}
    &=
      \CoordinateNormDual{\NormVector{ \SingularValues\np{\dualMatrix} }}{\LocalRank}
      \tag{as  $\CoordinateNormDual{\NormVector{\cdot}}{\LocalRank} =
      \supportFunction_{\lzero^{\leq\LocalRank}\cap\SPHERE_{\NormVector{\cdot}}}$ by~\eqref{eq:dual_coordinate_norm}}
      \eqfinp 
  \end{align*}
  Thus, we have proven that $ \CoordinateNormDualMatrix{\NormMatrix{\cdot}}{\LocalRank}=
  \CoordinateNormDual{\NormVector{\cdot}}{\LocalRank} \circ\SingularValues $, that is, the
  first equality in~\eqref{eq:rankmatrixnormstopnorm_b}.
  
  Now, it is easily established by~\eqref{eq:dual_coordinate_norm} that the vector
  norm~$ \CoordinateNormDual{\NormVector{ \cdot }}{\LocalRank} $ is a symmetric
  absolute norm (hence so is its dual norm $ \CoordinateNorm{\NormVector{\cdot}}{\LocalRank} $).
  As a consequence, the first equality in~\eqref{eq:rankmatrixnormstopnorm}
  easily follows by using~\eqref{eq:unitarilyinvariantnormequivalence_dual}
  (see \cite[Proposition IV.2.11]{Bhatia:1997}), 
  giving
  \begin{equation}
    \CoordinateNormMatrix{\NormMatrix{\cdot}}{\LocalRank}=
    \Bp{ \CoordinateNormDualMatrix{\NormMatrix{\cdot}}{\LocalRank} }_\star
    =  \Bp{ \CoordinateNormDual{\NormVector{\cdot}}{\LocalRank} \circ\SingularValues }_\star
    = \Bp{ \CoordinateNormDual{\NormVector{\cdot}}{\LocalRank} }_\star \circ\SingularValues
    = \CoordinateNorm{\NormVector{\cdot}}{\LocalRank} \circ\SingularValues 
    \eqfinp
  \end{equation}
  Thus, we have proven the   first equality in~\eqref{eq:rankmatrixnormstopnorm_a}.

  %
  There remains to prove the second equalities
  in~\eqref{eq:rankmatrixnormstopnorm_a}
  and \eqref{eq:rankmatrixnormstopnorm_b}.
  Because the symmetric absolute norm~$\NormVector{\cdot}$ in Proposition~\ref{pr:Links_with_coordinate_norms_and_the_lzeropseudonorm}
  is a symmetric monotonic norm~\cite[Theorem~2]{Bauer-Stoer-Witzgall:1961}, it is
  a so-called 
  orthant-monotonic norm~\cite{Gries:1967,Gries-Stoer:1967}.
  As a consequence, by \cite[Proposition~7]{Chancelier-DeLara:2022_SVVA}, 
  we get that, for any $ \LocalRank \in \ic{1,\minrowscolumns} $,
  $ \CoordinateNorm{\NormVector{\cdot}}{\LocalRank} = 
  \SupportDualNorm{\NormVector{\cdot}}{\LocalRank} $ 
  and 
  $ \CoordinateNormDual{\NormVector{\cdot}}{\LocalRank} =
  \TopDualNorm{\NormVector{\cdot}}{\LocalRank} $.
  This gives the second equalities
  in~\eqref{eq:rankmatrixnormstopnorm_a}
  and \eqref{eq:rankmatrixnormstopnorm_b}, and ends the proof. 

\end{proof}

As an illustration, when the source norm~$\NormMatrix{\cdot}$ is the \emph{Frobenius norm}
given by
\begin{equation}
  \NormMatrix{\primalMatrix}_{F}= 
  \sqrt{\Trace\np{\primalMatrix\transpose{\primalMatrix}}}
  =
  \NormVector{\SingularValues\np{\primalMatrix}}_{\ell_2}
  =
  \sqrt{\sum_{\LocalDiagonal=1}^{\minrowscolumns}\SingularValues_{\LocalDiagonal}\np{\primalMatrix}^2}
  \eqsepv\forall \primalMatrix\in \MATRIXdim{\rows}{\columns}
  \eqfinv
  \label{eq:Frobenius_norm}
\end{equation}
Equations~\eqref{eq:rankmatrixnormstopnorm_a} and
\eqref{eq:rankmatrixnormstopnorm_b} are
\begin{subequations}
  \begin{align}
    \CoordinateNormMatrix{\NormMatrix{\primalMatrix}_{F}}{\LocalRank}
    &=
      \LpSupportNorm{ \SingularValues\np{\primalMatrix} }{2}{\LocalRank} 
      \eqsepv
      \forall \LocalRank \in \ic{1,\minrowscolumns}
      \eqsepv\forall \primalMatrix\in \MATRIXdim{\rows}{\columns}
      \eqfinv
      \label{eq:Frobenius_norm_rankmatrixnormstopnorm_a}
    \\
    \CoordinateNormDualMatrix{\NormMatrix{\dualMatrix}_{F}}{\LocalRank}
    &=
      \LpTopNorm{ \SingularValues\np{\dualMatrix} }{2}{\LocalRank}
      =\sqrt{\sum\limits_{\LocalDiagonal=1}^{\LocalRank}\SingularValues_{\LocalDiagonal}\np{\dualMatrix}^2}
      \eqsepv
      \forall \LocalRank \in \ic{1,\minrowscolumns}
      \eqsepv\forall \dualMatrix\in \MATRIXdim{\rows}{\columns}
      \eqfinp
      \label{eq:Frobenius_norm_rankmatrixnormstopnorm_b}
  \end{align}
\end{subequations}

\section{\CAPRA-conjugacies and the rank function}
\label{CAPRA-conjugacies_and_the_rank_function}

In~\S\ref{CAPRA-couplings_and_conjugacies_for_matrices},
we adapt the definition of \Capra-couplings in
\cite{Chancelier-DeLara:2022_CAPRA_OPTIMIZATION}
to the case of matrices instead of vectors.
In~\S\ref{Variational_lower_bound_of_the_rank_function},
we provide a variational lower bound of the rank function.

\subsection{\CAPRA-couplings and conjugacies for matrices}
\label{CAPRA-couplings_and_conjugacies_for_matrices}

We adapt the definition of \Capra-couplings in
\cite{Chancelier-DeLara:2022_CAPRA_OPTIMIZATION} 
to the space~$\MATRIXdim{\rows}{\columns}$ of matrices. 

\begin{definition}
  \label{de:coupling_CAPRA}
  Let $\NormMatrix{\cdot}$ be a source matrix norm on~$\MATRIXdim{\rows}{\columns}$.
  The \Capra-coupling~$\couplingCapraNormMatrix$,
  between $\MATRIXdim{\rows}{\columns}$ and $\MATRIXdim{\rows}{\columns}$,
  associated with~$\NormMatrix{\cdot}$, is defined by: 
  \begin{equation}
    \forall \primalMatrix,\dualMatrix\in\MATRIXdim{\rows}{\columns} \eqsepv
    \couplingCapraNormMatrix\np{\primalMatrix,\dualMatrix} = 
    \begin{cases}
      \frac{\Trace\np{\primalMatrix\transpose{\dualMatrix}}}{\NormMatrix{\primalMatrix}}
      & \text{ if } \primalMatrix\neq 0 \eqfinv 
      \\
      0 & \text{ otherwise.}
    \end{cases}
    \label{eq:coupling_CAPRA}                      
  \end{equation}
  For any function $ \fonctionprimalMatrix : \MATRIXdim{\rows}{\columns} \to \barRR
  $, the \emph{$\couplingCapraNormMatrix$-Fenchel-Moreau conjugate},
  or \emph{\Capra-conjugate}, is
  the function $ \SFM{\fonctionprimalMatrix}{\couplingCapraNormMatrix} : \MATRIXdim{\rows}{\columns} \to \barRR $ 
  defined by
  \begin{align}
    \SFM{\fonctionprimalMatrix}{\couplingCapraNormMatrix}\np{\dualMatrix}
    &= 
      \sup_{\primalMatrix \in \MATRIXdim{\rows}{\columns}} \Bp{ \couplingCapraNormMatrix\np{\primalMatrix,\dualMatrix} 
      \LowPlus \bp{ -\fonctionprimalMatrix\np{\primalMatrix} } } 
      \eqsepv \forall \dualMatrix \in \MATRIXdim{\rows}{\columns}
      \eqfinv
      \label{eq:Fenchel-Moreau_conjugate}
      \intertext{and the \emph{$\couplingCapraNormMatrix$-Fenchel-Moreau biconjugate},
      or \emph{\Capra-biconjugate},  is
      the function $ \SFMbi{\fonctionprimalMatrix}{\couplingCapraNormMatrix} : \MATRIXdim{\rows}{\columns} \to \barRR $ 
      defined by}
      \SFMbi{\fonctionprimalMatrix}{\couplingCapraNormMatrix}\np{\primalMatrix}
    &= 
      \sup_{\dualMatrix \in \MATRIXdim{\rows}{\columns}} \Bp{ \couplingCapraNormMatrix\np{\primalMatrix,\dualMatrix} 
      \LowPlus \bp{ -\SFM{\fonctionprimalMatrix}{\couplingCapraNormMatrix}\np{\dualMatrix} } } 
      \eqsepv \forall \primalMatrix \in \MATRIXdim{\rows}{\columns}
      \eqfinp 
      \label{eq:Fenchel-Moreau_biconjugate}
  \end{align}
\end{definition}

Then, we show below that the \Capra-conjugate and biconjugate of the rank function
are expressed in function of the generalized dual $\LocalRank$-rank matrix norms $
\sequence{\CoordinateNormDualMatrix{\NormMatrix{\cdot}}{\LocalRank}}{\LocalRank \in\ic{1,\minrowscolumns}} $
and the generalized $\LocalRank$-rank matrix norms $
\sequence{\CoordinateNormMatrix{\NormMatrix{\cdot}}{\LocalRank}}{\LocalRank \in\ic{1,\minrowscolumns}} $, 
given by Definition~\ref{de:CoordinateNormMatrix}.
We do not give the proof as it is a simple adaptation, to the matrix case, of the proofs of
\cite[Propositions~4.4, 4.5]{Chancelier-DeLara:2022_CAPRA_OPTIMIZATION}.

\begin{proposition}
  \label{pr:rank_biconjugate_varphi}
  Let $\NormMatrix{\cdot}$ be a source matrix norm on $\MATRIXdim{\rows}{\columns}$,
  and $\couplingCapraNormMatrix$ be the associated \Capra-coupling as in
  Definition~\ref{de:coupling_CAPRA}. 

  For any function $ \varphi : \ic{0,\minrowscolumns} \to \barRR $, we have
  that (with the convention that $ \CoordinateNormDualMatrix{\NormMatrix{\cdot}}{0}=0$)
  \begin{align}
    \SFM{ \np{\varphi\circ\rank} }{\couplingCapraNormMatrix}\np{\dualMatrix}
    &=
      \sup_{\LocalDiagonal\in\ic{0,\minrowscolumns}}
      \Ba{
      \CoordinateNormDualMatrix{\NormMatrix{\dualMatrix}}{\LocalDiagonal}-\varphi\np{\LocalDiagonal} 
      }
      \eqsepv \forall  \dualMatrix   \in \MATRIXdim{\rows}{\columns}                                                         
      \eqfinv
      \label{eq:conjugate_rank_varphi}                                                                          
      \intertext{and, for any function $ \varphi : \ic{0,\minrowscolumns} \to \RR_+ $
      (that is, with nonnegative finite values)
      and such that $ \varphi\np{0}=0 $, we have that}
      \SFMbi{\np{\varphi\circ\rank}}{\couplingCapraNormMatrix}\np{\primalMatrix}
    &=
      \frac{ 1 }{ \NormMatrix{\primalMatrix} } 
      \min_{ \substack{%
      \primalMatrix^{(1)} \in \MATRIXdim{\rows}{\columns}, \ldots,
      \primalMatrix^{(\minrowscolumns)} \in \MATRIXdim{\rows}{\columns}
    \\
    \sum_{ \LocalRank=1 }^{ \minrowscolumns } \CoordinateNormMatrix{\NormMatrix{\primalMatrix^{(\LocalRank)}}}{\LocalRank}
    \leq \NormMatrix{\primalMatrix}
    \\
    \sum_{ \LocalRank=1 }^{ \minrowscolumns } \primalMatrix^{(\LocalRank)} = \primalMatrix} }
    \sum_{ \LocalRank=1 }^{ \minrowscolumns }\varphi\np{\LocalRank}
    \CoordinateNormMatrix{\NormMatrix{\primalMatrix^{(\LocalRank)}}}{\LocalRank} 
    \eqsepv \forall \primalMatrix \in \MATRIXdim{\rows}{\columns}\setminus\{0\} 
    \eqfinp
    \label{eq:biconjugate_rank_varphi}
  \end{align}
\end{proposition}

\subsection{Variational lower bound and expression of the rank function}
\label{Variational_lower_bound_of_the_rank_function}

Now, thanks to Proposition~\ref{pr:rank_biconjugate_varphi}, we obtain a
variational lower bound of the rank function,
and also a variational expression when the source norm is the Frobenius norm. 

\begin{theorem}
  Let $\NormMatrix{\cdot}$ be a source norm on the space
  $\MATRIXdim{\rows}{\columns}$ of matrices,
  with associated sequence $
  \sequence{\CoordinateNormMatrix{\NormMatrix{\cdot}}{\LocalRank}}{\LocalRank \in\ic{1,\minrowscolumns}} $ of generalized $\LocalRank$-rank matrix
  norms as in Definition~\ref{de:CoordinateNormMatrix}. 
  Then,   we have the following variational lower bound of the rank function
  \begin{equation}
    \rank\np{\primalMatrix} \geq
    \frac{ 1 }{ \NormMatrix{\primalMatrix} }       \min_{ \substack{%
        \primalMatrix^{(1)} \in \MATRIXdim{\rows}{\columns}, \ldots,
        \primalMatrix^{(\minrowscolumns)} \in \MATRIXdim{\rows}{\columns}
        \\
        \sum_{ \LocalRank=1 }^{ \minrowscolumns } \CoordinateNormMatrix{\NormMatrix{\primalMatrix^{(\LocalRank)}}}{\LocalRank}
        \leq \NormMatrix{\primalMatrix}
        \\
        \sum_{ \LocalRank=1 }^{ \minrowscolumns } \primalMatrix^{(\LocalRank)} = \primalMatrix} }
    \sum_{ \LocalRank=1 }^{ \minrowscolumns }\LocalRank
    \CoordinateNormMatrix{\NormMatrix{\primalMatrix^{(\LocalRank)}}}{\LocalRank} 
    \eqsepv \forall \primalMatrix \in \MATRIXdim{\rows}{\columns}\setminus\{0\} 
    \eqfinp
    \label{eq:variational_lower_bound_of_the_rank_function}      
  \end{equation}
  Moreover, if the source norm is the {Frobenius norm}~$\NormMatrix{\cdot}_{F}$
  given by~\eqref{eq:Frobenius_norm}, 
  the inequality in~\eqref{eq:variational_lower_bound_of_the_rank_function} is
  an equality:
   \begin{equation}
    \rank\np{\primalMatrix} =
    \frac{ 1 }{ \NormMatrix{\primalMatrix}_{F} }       \min_{ \substack{%
        \primalMatrix^{(1)} \in \MATRIXdim{\rows}{\columns}, \ldots,
        \primalMatrix^{(\minrowscolumns)} \in \MATRIXdim{\rows}{\columns}
        \\
        \sum_{ \LocalRank=1 }^{ \minrowscolumns } \CoordinateNormMatrix{\NormMatrix{\primalMatrix^{(\LocalRank)}}_{F}}{\LocalRank}
        \leq \NormMatrix{\primalMatrix}_{F}
        \\
        \sum_{ \LocalRank=1 }^{ \minrowscolumns } \primalMatrix^{(\LocalRank)} = \primalMatrix} }
    \sum_{ \LocalRank=1 }^{ \minrowscolumns }\LocalRank
    \CoordinateNormMatrix{\NormMatrix{\primalMatrix^{(\LocalRank)}}_{F}}{\LocalRank} 
    \eqsepv \forall \primalMatrix \in \MATRIXdim{\rows}{\columns}\setminus\{0\} 
    \eqfinp
    \label{eq:variational_rank_function}      
  \end{equation}
  \label{th:variational_lower_bound_of_the_rank_function}
\end{theorem}

\begin{proof}
  From the expression~\eqref{eq:biconjugate_rank_varphi} of
  $ \SFMbi{\rank}{\couplingCapraNormMatrix}$,
  with $\varphi$ the identity function, and from the (true for any coupling)
  inequality 
  $ \rank \geq \SFMbi{\rank}{\couplingCapraNormMatrix}$, we readily
  deduce~\eqref{eq:variational_lower_bound_of_the_rank_function}.
  
  In the rest of the proof --- which follows that of
  \cite[Theorem~3.5]{Chancelier-DeLara:2021_ECAPRA_JCA} ---
  $\NormMatrix{\cdot}$ denotes the Frobenius norm~\eqref{eq:Frobenius_norm}
  (instead of $\NormMatrix{\cdot}_{F}$ to alleviate notation).
  We consider a fixed matrix $ \primalMatrix \in
  \MATRIXdim{\rows}{\columns}\setminus\{0\} $
  and we are going to show that
  $ \rank\np{\primalMatrix} =
  \SFMbi{\rank}{\couplingCapraNormMatrix}\np{\primalMatrix} $.
  We denote by $\LocalRank=\rank\np{\primalMatrix} \geq 1 $ the rank
  of~$\primalMatrix$.
  By the factorization Equation~\eqref{eq:Frobenius_norm_rankmatrixnormstopnorm_b}, and as $
  \SingularValues_{\LocalDiagonal}\np{\primalMatrix}=0 \iff \LocalDiagonal > \LocalRank$, we have that
  \begin{equation}
    \CoordinateNormDualMatrix{\NormMatrix{\primalMatrix}}{\LocalIndex}
    =\LpTopNorm{ \SingularValues\np{\primalMatrix} }{2}{\LocalIndex}
    =\sqrt{\sum\limits_{\LocalDiagonal=1}^{\LocalIndex}\SingularValues_{\LocalDiagonal}\np{\primalMatrix}^2}
    \quad
    \begin{cases}
      &=
        \sqrt{\sum\limits_{\LocalDiagonal=1}^{\LocalRank}\SingularValues_{\LocalDiagonal}\np{\primalMatrix}^2} 
        =\NormMatrix{\primalMatrix}
        \eqsepv \forall \LocalIndex \geq \LocalRank
        \eqfinv 
      \\
      &<
        \sqrt{\sum\limits_{\LocalDiagonal=1}^{\LocalRank}\SingularValues_{\LocalDiagonal}\np{\primalMatrix}^2} 
        =\NormMatrix{\primalMatrix}
        \eqsepv \forall \LocalIndex < \LocalRank
        \eqfinp
    \end{cases}  
    \label{eq:variational_lower_bound_of_the_rank_function_proof}
  \end{equation}
  We consider the function $\phi: ]0,+\infty[ \to \RR$ defined by
  \begin{equation}
    \phi(\lambda) =
    \frac{ \Trace\np{\lambda \primalMatrix\transpose{\primalMatrix}} }{ \NormMatrix{\primalMatrix} }
    - \sup_{\LocalIndex\in\ic{0,\minrowscolumns}} 
    \Ba{ \CoordinateNormDualMatrix{\NormMatrix{\lambda\primalMatrix}}{\LocalIndex} - \LocalIndex
    } 
    \eqsepv \forall \lambda > 0 \eqfinv
    \label{eq:phi}
  \end{equation}
  and we will show that 
  $ \lim_{\lambda \to +\infty} \phi(\lambda) =\LocalRank $.
  We have that
  \begin{align*}
    \phi(\lambda) 
    &=
      \lambda \NormMatrix{\primalMatrix}
      - \sup\Bp{0,\sup_{\LocalIndex\in\ic{1,\minrowscolumns}} 
      \Ba{ \CoordinateNormDualMatrix{\lambda\NormMatrix{\primalMatrix}}{\LocalIndex} - \LocalIndex
      } }
      \intertext{ by definition~\eqref{eq:phi} of~$\phi$,
      by the convention that $
      \CoordinateNormDualMatrix{\NormMatrix{\primalMatrix}}{0}=0 $ and by 
      $ \NormMatrix{\primalMatrix}^2 = \Trace\np{\primalMatrix\transpose{\primalMatrix}} $ }
    &= 
      \lambda \CoordinateNormDualMatrix{\NormMatrix{\primalMatrix}}{\LocalRank}
      + \inf \Ba{0, -\sup_{\LocalIndex\in\ic{1,\minrowscolumns}} 
      \Bc{ \lambda \CoordinateNormDualMatrix{\NormMatrix{\primalMatrix}}{\LocalIndex} - \LocalIndex } }
      \tag{as
      $\NormMatrix{\primalMatrix}=\CoordinateNormDualMatrix{\NormMatrix{\primalMatrix}}{\LocalRank}$
      by~\eqref{eq:variational_lower_bound_of_the_rank_function_proof}}
    \\
    &= 
      \inf \Ba{ \lambda \CoordinateNormDualMatrix{\NormMatrix{\primalMatrix}}{\LocalRank} ,
      \lambda \CoordinateNormDualMatrix{\NormMatrix{\primalMatrix}}{\LocalRank} + 
      \inf_{\LocalIndex\in\ic{1,\minrowscolumns}} \Bp{ -\Bc{ \lambda \CoordinateNormDualMatrix{\NormMatrix{\primalMatrix}}{\LocalIndex} - \LocalIndex } } }
    \\
    &= 
      \inf \Ba{ \lambda \CoordinateNormDualMatrix{\NormMatrix{\primalMatrix}}{\LocalRank} ,
      \inf_{\LocalIndex\in\ic{1,\minrowscolumns}} \Bp{ \lambda \bp{ \CoordinateNormDualMatrix{\NormMatrix{\primalMatrix}}{\LocalRank} -
      \CoordinateNormDualMatrix{\NormMatrix{\primalMatrix}}{\LocalIndex} } + \LocalIndex } }
    \\
    &= 
      \inf
      \Big\{ \lambda \CoordinateNormDualMatrix{\NormMatrix{\primalMatrix}}{\LocalRank} ,
      \inf_{\LocalIndex\in\ic{1,\LocalRank-1}} \Bp{ \lambda \bp{ \CoordinateNormDualMatrix{\NormMatrix{\primalMatrix}}{\LocalRank} -
      \CoordinateNormDualMatrix{\NormMatrix{\primalMatrix}}{\LocalIndex} } + \LocalIndex } ,
    \\
    &
      \hspace{3cm} \inf_{\LocalIndex\in\ic{\LocalRank,\minrowscolumns}} \Bp{ \lambda \bp{ \CoordinateNormDualMatrix{\NormMatrix{\primalMatrix}}{\LocalRank} -
      \CoordinateNormDualMatrix{\NormMatrix{\primalMatrix}}{\LocalIndex} } + \LocalIndex }
      \Big\}
    \\
    &= 
      \inf \Ba{ \lambda \CoordinateNormDualMatrix{\NormMatrix{\primalMatrix}}{\LocalRank} ,
      \inf_{\LocalIndex\in\ic{1,\LocalRank-1}} \Bp{ \lambda \bp{ \CoordinateNormDualMatrix{\NormMatrix{\primalMatrix}}{\LocalRank} -
      \CoordinateNormDualMatrix{\NormMatrix{\primalMatrix}}{\LocalIndex} } + \LocalIndex }, \LocalRank }
  \end{align*}
  as $\CoordinateNormDualMatrix{\NormMatrix{\primalMatrix}}{\LocalIndex}=
  \CoordinateNormDualMatrix{\NormMatrix{\primalMatrix}}{\LocalRank}$ for
  $\LocalIndex \geq \LocalRank$
  by~\eqref{eq:variational_lower_bound_of_the_rank_function_proof}.
  Let us show that the two first terms in the infimum
  go to $+\infty$ when $ \lambda \to +\infty $.
  The first term $ \lambda \CoordinateNormDualMatrix{\NormMatrix{\primalMatrix}}{\LocalRank} $ goes to $+\infty$ 
  because $ \CoordinateNormDualMatrix{\NormMatrix{\primalMatrix}}{\LocalRank}=\NormMatrix{\primalMatrix}>0 $ by assumption
  ($\primalMatrix \neq 0$).
  The second term
  $ \inf_{\LocalIndex\in\ic{1,\LocalRank-1}}
  \Bp{ \lambda \bp{ \CoordinateNormDualMatrix{\NormMatrix{\primalMatrix}}{\LocalRank} - %
      \CoordinateNormDualMatrix{\NormMatrix{\primalMatrix}}{\LocalIndex}}} + \LocalIndex  $
  also goes to $+\infty$ because 
  $\rank\np{\primalMatrix}=\LocalRank $, so that 
  $ \NormMatrix{\primalMatrix}=\CoordinateNormDualMatrix{\NormMatrix{\primalMatrix}}{\LocalRank} > 
  \CoordinateNormDualMatrix{\NormMatrix{\primalMatrix}}{\LocalIndex} $
  for $\LocalIndex\in\ic{1,\LocalRank-1}$
  as shown in~\eqref{eq:variational_lower_bound_of_the_rank_function_proof}.
  Therefore, we get that 
  $ \lim_{\lambda \to +\infty} \phi(\lambda) = 
  \inf \{ +\infty, +\infty, \LocalRank \}= \LocalRank $.
  This concludes the proof since 
  \begin{align*}
    \LocalRank = \lim_{\lambda \to +\infty} \phi(\lambda) 
    & \leq 
      \sup_{\dualMatrix \in \MATRIXdim{\rows}{\columns} } \bgp{ \frac{ \Trace\np{\primalMatrix\transpose{\dualMatrix}} }{ \NormMatrix{\primalMatrix} }
      - 
      \sup_{\LocalIndex\in\ic{0,\minrowscolumns}} 
      \Ba{ \CoordinateNormDualMatrix{\NormMatrix{\dualMatrix}}{\LocalIndex} - \LocalIndex } }  
      \tag{ by definition~\eqref{eq:phi} of~$\phi$ }
    \\
    &=
      \sup_{\dualMatrix \in \MATRIXdim{\rows}{\columns} } \bgp{ \frac{\Trace\np{\primalMatrix\transpose{\dualMatrix}}}{ \NormMatrix{\primalMatrix} }
      \ - \SFM{\rank}{\couplingCAPRA}\np{\dualMatrix} }
      \intertext{by the formula~\eqref{eq:conjugate_rank_varphi}                                                                          
      for the conjugate $ \SFM{ \rank }{\couplingCAPRA} $ }
    &=
      \SFMbi{ \rank }{\couplingCAPRA}\np{\primalMatrix}
      \tag{ by the biconjugate formula~\eqref{eq:Fenchel-Moreau_biconjugate} }
    \\
    & \leq 
      \rank\np{\primalMatrix} 
      \tag{ as
      $ \SFMbi{ \rank }{\couplingCAPRA} \leq \rank $}
    \\
    & = \LocalRank 
      \tag{ by assumption }
      \eqfinp
  \end{align*}
  Therefore, we have obtained that $ \LocalRank=\SFMbi{ \rank
  }{\couplingCAPRA}\np{\primalMatrix}=\rank\np{\primalMatrix} $.

  This ends the proof.
\end{proof}

\section{Conclusion}

In this paper, we have shown how 
to obtain a 
variational lower bound of the rank function
(Theorem~\ref{th:variational_lower_bound_of_the_rank_function}).
Interestingly, the formula depends on a (source) matrix norm
and on the derived generalized $\LocalRank$-rank matrix norms, that we introduce
(Definition~\ref{de:CoordinateNormMatrix}).
This is made possible by the versatility of the \Capra-couplings,
themselves depending on a matrix norm (Definition~\ref{de:coupling_CAPRA}).
Moreover, we show that the variational expression we obain is equal to the rank function
when the source norm is the Frobenius norm (Theorem~\ref{th:variational_lower_bound_of_the_rank_function}).

Thus, we hope to offer a general framework to derive matrix norms suitable for
optimization problems involving the rank function,
as well as variational formulations.

\newcommand{\noopsort}[1]{} \ifx\undefined\allcaps\def\allcaps#1{#1}\fi

\end{document}